\documentclass[12 pt]{article}%
\usepackage{amsmath, amsfonts, amsthm, color,latexsym}
\usepackage{amsmath}
\usepackage{amsfonts}
\usepackage{amssymb}
\usepackage{color, soul}
\usepackage[all]{xy}
\usepackage{graphicx}%
\providecommand{\U}[1]{\protect\rule{.1in}{.1in}}
\allowdisplaybreaks[4]
\newtheorem{theorem}{Theorem}[section]
\newtheorem{proposition}[theorem]{Proposition}

\newtheorem{example}[theorem]{Example}

\newtheorem{final remark}[theorem]{Final Remark}
\newtheorem{definition}[theorem]{Definition}
\allowdisplaybreaks[4]

\begin{document}

\title{Hyper-ideals of multilinear operators and two-sided polynomial ideals generated by sequence classes}
\author{Geraldo Botelho\thanks{Supported by CNPq Grant
304262/2018-8 and Fapemig Grant PPM-00450-17.}\,\, and Raquel Wood\thanks{Supported by a CAPES scholarship.\newline 2020 Mathematics Subject Classification: 46B45, 47H60, 46G25, 47L22, 47B10.\newline Keywords: Banach sequence spaces, polynomial ideals, multi-ideals, operator ideals.
}}
\date{}
\maketitle

\begin{abstract} In the nonlinear field of multilinear operators and homogeneous polynomials between Banach spaces, we develop a technique, based on the transformation of vector-valued sequences, to create new examples of hyper-ideals of multilinear operators, polynomial hyper-ideals and polynomial two-sided ideals. Several well studied ideals are shown to be actually hyper-ideals or two-sided ideals.
\end{abstract}

\section*{Introduction}

Ideals of multilinear operators (multi-ideals) and ideals of homogeneous polynomials (polynomial ideals) between Banach spaces have been intensively investigated since the seminal 1983 paper by Pietsch \cite{pietsch83}  (see, e.g., the references in \cite{davidson, ewertonjmaa}). The basic idea is the stability of the class with respect to the composition with linear operators, for example: a class $\cal Q$ of homogeneous polynomials is an ideal if the following holds: in the following chain of mappings between Banach spaces
$$G \stackrel{u}{\longrightarrow} E \stackrel{P}{\longrightarrow} F \stackrel{t}{\longrightarrow} H, $$
if $P$ is a continuous homogeneous polynomial belonging to $\cal Q$ and $u$ and $t$ are bounded linear operators, then the composition $t \circ P \circ u$ belongs to $\cal Q$. Several distinguished classes of multilinear operators and of polynomials are ideals, such as compact/weakly compact mappings and a number of nuclear-type and absolutely summing type-mappings.

Special types of multi-ideals and of polynomial ideals were introduced and developed in \cite{ewertonlaa, ewertonlama, ewertonjmaa, raquelsegundo, samuel, popalaa2016, velanga} taking into account the composition with multilinear operators/homogeneous polynomials instead of only with linear operators. More precisely:\\
$\bullet$ A class of multilinear operators which is stable with respect to the composition with multilinear operators on the left-hand side and with linear operators on the right-hand side are called {\it hyper-ideals of multilinear opeators.}\\
$\bullet$ A class of homogeneous polynomials which is stable with respect to the  composition with homogeneous polynomials on the left-hand side and with linear operators on the right-hand side are called {\it polynomial hyper-ideals}.
\\
$\bullet$ A class of homogeneous polynomials which is stable with respect to the  composition with homogeneous polynomials on both the left-hand and right-hand sides are called {\it polynomial two-sided ideals}.

Precise definitions are given in Section \ref{section 1}. Illustrative examples: the classes of compact and weakly compact multilinear operators are hyper-ideals, the class of compact homogeneous polynomials is a two-sided polynomial ideal and the class of weakly compact polynomials is a polynomial hyper-ideal that fails to be a two-sided ideal.

Techniques to create such special types of ideals were developed in the references quoted above. The goal of this paper is to developed a new technique to create hyper and two-sided ideals that, although quite natural, has gone unnoticed. As a consequence, several well studied ideals are shown to be hyper-ideals or two-sided ideals. This new technique is based on the notion of sequence classes, a concept that was introduced in \cite{jamilsonmonat} and has proved to be quite fruitful: applications and new developments can be found in \cite{achourmediterranean, jamilsoncomplut, pacific, davidson, jamilsonjoedsoncolloq, joilsonfabriciomed, baianosmfat}.

The paper is organized as follows. In Section \ref{section 1} we give the definitions that are needed to develop our new technique. In Sections \ref{section 2} and \ref{section 3} we show how sequence classes can be used to create new hyper-ideals of multilinear operators and new polynomial hyper and two-sided ideals, respectively. In the final Section \ref{section 4} we show that the ideals our technique provides are not composition ideals, establishing that the hyper-ideals and two-sided ideals we obtain are new indeed. %  our technique does not create composition ideals, showing that

\section{Basic concepts} \label{section 1}
Throughout the paper, $n$ is a natural number, $E,E_1, \ldots, E_n,F,G, H$ are  Banach spaces over $\mathbb{K} = \mathbb{R}$ or $\mathbb{C}$, ${\cal L}(E_1, \ldots, E_n;F)$ denotes the Banach space of continuous $n$-linear operators from $E_1 \times \cdots \times E_n$ to $F$ and ${\cal P}(^nE;F)$ denotes the Banach space of continuous $n$-homogeneous polynomials from $E$ to $F$, both of them endowed with their usual norms. If $E_1 = \cdots = E_n = E$,  we simply write ${\cal L}(^nE;F)$. The closed unit ball of $E$ is denoted by $B_E$ and its topological dual by $E^*$. Let $\varphi \in E^*, \varphi_1 \in E_1^*, \ldots, \varphi_n \in E_n^*$ and   $b \in F$ be given. Linear combinations of $n$-linear operators of the type
 $$\varphi_1 \otimes \cdots \otimes \varphi_n \otimes b \colon E_1 \times \cdots \times E_n \longrightarrow F~,~(\varphi_1 \otimes \cdots \otimes \varphi_n \otimes b)(x_1, \ldots, x_n) = \varphi_1(x_1) \cdots \varphi_n(x_n)b, $$
 are called multilinear operators of finite type. And linear combinations of $n$-homogeneous polynomials of the type
 $$\varphi^n \otimes b \colon E \longrightarrow F~,~(\varphi^n \otimes b)(x) = \varphi(x)^n b, $$
 are called polynomials of finite type. For the general theory of multilinear operators and homogeneous polynomials we refer to \cite{dineen, mujica} and for operator ideals to \cite{df, livropietsch}.

\begin{definition}\rm \cite{ewertonlaa}  A {\it Banach hyper-ideal of multilinear operators}, or simply a {\it Banach hyper-ideal}, is a pair $({\cal H}, \|\cdot\|_{\cal H})$ where  $\cal H$ is a subclass of  the class of continuous multilinear operators between Banach spaces and $\|\cdot\|_{\cal H} \colon {\cal H} \longrightarrow \mathbb{R}$ is a function such that, for all $n \in \mathbb{N}$ and Banach spaces $E_{1},\ldots,E_{n},F$, the components
$$\mathcal{H}(E_{1},\ldots,E_{n};F) := \mathcal{L}(E_{1},\ldots,E_{n};F) \cap \mathcal{H}$$
fulfill the following conditions:
\newline
$\bullet$ $\mathcal{H}(E_{1},\ldots,E_{n};F)$ is a linear subspace of $\mathcal{L}(E_{1},\ldots,E_{n};F)$ containing the multilinear operators of finite type, on which $\|\cdot\|_{\cal H}$ is a complete norm satisfying
$$\|I_{n}\colon  \mathbb{K}^{n} \longrightarrow \mathbb{K}~,~I_{n}(\lambda_{1},\ldots,\lambda_{n}) = \lambda_{1} \cdots \lambda_{n}\|_{\cal H} = 1.$$
$\bullet$ The hyper-ideal property: given natural numbers $n$ and $1 \leq m_{1} < \cdots < m_{n}$, Banach spaces $G_{1},\ldots,G_{m_{n}}$, $E_{1},\ldots,E_{n}$,$F,H$, if $A \in \mathcal{H}(E_{1},\ldots,E_{n};F)$,
$B_{1} \in \mathcal{L}(G_{1},\ldots,G_{m_{1}};E_{1})$, $\ldots,B_{n} \in \mathcal{L}(G_{m_{n-1} + 1},\ldots,G_{m_{n}};E_{n})$ and $t \in \mathcal{L}(F;H)$, then $t \circ A \circ (B_{1},\ldots,B_{n})$ belongs to  $\mathcal{H}(G_{1},\ldots,G_{m_{n}};H)$ and
$$\|t \circ A \circ (B_{1},\ldots,B_{n})\|_{\mathcal{H}} \leq \|t\| \cdot \|A\|_{\mathcal{H}} \cdot \|B_{1}\| \cdots \|B_{n}\|.$$
\end{definition}

If the hyper-ideal property holds only for $m_1 = 1, m_2 = 2, \ldots, m_n = n$, that is, if the multilinear operators $B_1, \ldots, B_n$ are replaced with linear operators, then we recover the classical notion of Banach ideal of multilinear operators, or Banach multi-ideal (see \cite{fg}).

\begin{definition}\rm Let $\cal Q$ be a subclass of the class of homogeneous polynomials between Banach spaces and $\|\cdot\|_{\cal Q} \colon {\cal Q} \longrightarrow \mathbb{R}$ be a function such that, for all $n \in \mathbb{N}$ and Banach spaces $E,F$, the component
$${\cal Q}(^nE;F) := {\cal P}(^nE;F ) \cap {\cal Q}$$
is a linear subspace of ${\cal P}(^nE;F)$ containing the polynomials of finite type on which $\|\cdot\|_{\cal Q}$ is a complete norm satisfying
$$\|\widehat{I}_{n}\colon  \mathbb{K} \longrightarrow \mathbb{K}~,~\widehat{I}_{n}(\lambda) = \lambda^n\|_{\cal Q} = 1.$$
$\bullet$ \cite{ewertonarxiv, raquelsegundo} Let $(C_{n})_{n=1}^{\infty}$ be a sequence of positive real numbers such that $C_{n} \geq 1$ for every $n \in \mathbb{N}$ and $C_{1} =1$. We say that  $(\mathcal{Q}, \|\cdot\|_{\cal Q})$ is a {\it $(C_{n})_{n=1}^{\infty}$-polynomial Banach hyper-ideal} if the following hyper-ideal property is satisfied:  for $n,m \in \mathbb{N}$, Banach spaces $E, F, G$ and $H$, if $P \in \mathcal{Q}(^{n}E;F)$, $Q \in \mathcal{P}(^{m}G;E)$ and $t \in \mathcal{L}(F;H)$, then $t \circ P \circ Q \in \mathcal{Q}(^{mn}G;H)$ and
$$\|t \circ P \circ Q \|_{\mathcal{Q}} \leq C_m^m \cdot \|t\| \cdot \|P\|_{\mathcal{Q}} \cdot \|Q\|^{n}.$$
If $C_{n} = 1$ for every $n \in \mathbb{N}$, we simply say that $(\mathcal{Q}, \|\cdot\|_{\cal Q})$ is a {\it polynomial Banach hyper-ideal}.

When the hyper-ideal property holds for every $n \in \mathbb{N}$ but only for  $m=1$, the classical concept of Banach polynomial ideal  is recovered (see, e.g., \cite{fg}).\\
$\bullet$ \cite{ewertonjmaa} Let $(C_n,K_n)_{n=1}^\infty$ be a sequence of pairs of positive real numbers with $C_n,K_n \geq 1$ for every $n$ and $C_1 = K_1 = 1$. We say that $({\cal Q}, \|\cdot\|_{\cal Q})$ is a {\it $(C_n,K_n)_{n=1}^\infty$-polynomial Banach two-sided ideal} if the following holds: for  $m,n,r \in \mathbb{N}$, Banach spaces $E$,$F$, $G$ and $H$, if $P \in {\cal Q}(^nE;F)$, $Q \in {\cal P}(^mG;E)$ and $R \in {\cal P}(^rF;H)$, then $R \circ P \circ Q \in {\cal Q}(^{rmn}G;H)$ and
$$\|R \circ P \circ Q\|_{\cal Q} \leq K_r\cdot C_m^{rn}\cdot\|R\| \cdot \|P\|_{\cal Q}^r  \cdot \|Q\|^{rn}. $$
\end{definition}

To define sequence classes we need some terminology. By $(e_j)_{j=1}^\infty$ we denote the canonical vectors of the spaces of scalar-valued sequences. %denotamos os vetores canônicos dos espaços de sequências escalares, isto é, $e_1 = (1,0,0,0, \ldots), \,e_2 = (0,1,0,0,0, \ldots), \ldots$.
 The symbol $E \hookrightarrow F$ means that the Banach space $E$ is a linear subspace of the Banach space $F$ and that $\|x\|_F \leq \|x\|_E$ for every $x \in E$. %Given a sequence $(x_j)_{j=1}^\infty$ in a linear space and $k \in \mathbb{N}$, we define \textcolor{red}{(Acho que n\~ao usa)}
%$$(x_j)_{j=1}^k := (x_1, x_2, \ldots, x_k,0,0,0, \ldots).  $$
As usual, $c_{00}(E)$ stands for the linear space of eventually null $E$-valued seuences and $\ell_\infty(E)$ for the Banach space of $E$-valued bounded sequences with the supremum norm.

\begin{definition}\rm \cite{jamilsonmonat}
A {\it sequence class} is a rule $E \mapsto X(E)$ that assigns to each Banach space $E$ a Banach space $X(E)$ formed by $E$-valued sequences (of course we are considering the usual coordinatewise algebraic operations) such that %ormado por sequências em $E$, isto é, um subespaço vetorial de $E^{\mathbb{N}}$ com as operações usuais, satisfazendo as seguintes condições:
$c_{00}(E) \subseteq X(E)$, $X(E) \hookrightarrow \ell_{\infty}(E)$ and  $\|e_{j}\|_{X(\mathbb{K})} = 1$ for every $ j \in \mathbb{N}$.

%Uma classe de sequências $X$ é dita {\it finitamente determinada} \textcolor{red}{(Acho que n\~ao usa)} se para toda sequência $(x_{j})_{j=1}^{\infty} \in E^{\mathbb{N}}$, $(x_{j})_{j=1}^{\infty} \in X(E)$ se, e somente se,
%$\sup\limits_{k} \|(x_{j})_{j=1}^{k}\|_{X(E)} < \infty,$ e, neste caso,
%$$\|(x_{j})_{j=1}^{\infty}\|_{X(E)} = \sup\limits_{k \in \mathbb{N}} \|(x_{j})_{j=1}^{k}\|_{X(E)}.$$

A sequence class $X$ is {\it linearly stable} if for all Banach spaces $E$ and $F$ and any operator $u \in \mathcal{L}(E;F)$, it holds $(u(x_{j}))_{j=1}^{\infty} \in X(F)$ whenever $(x_{j})_{j=1}^{\infty} \in X(E)$ and, in this case, the induced linear operator
$$\widetilde{u} \colon X(E) \longrightarrow X(F)~,~\widetilde{u}\left((x_j)_{j=1}^\infty \right) = (u(x_j))_{j=1}^\infty ,$$
is continuous with $\|\widetilde{u}\| = \|u\|$.
\end{definition}

Plenty of examples of linearly stable sequences classes can be found in the references quoted in the introduction. We just mention some well known classes: the class $E \mapsto c_0(E)$ of norm null sequences, the classes $E \mapsto \ell_p(E)$, $E \mapsto \ell_p^w(E), E \mapsto \ell_p\langle E \rangle, 1 \leq p < \infty$, of absolutely, weakly and strongly Cohen $p$-summable sequences, the class $E \mapsto \ell_\infty(E)$ of bounded sequences and the class $E \mapsto {\rm Rad}(E)$ of almost unconditionally summable sequences.

We also need the following H\"older-type compatibility condition: for sequence classes $X_{1},\ldots,X_{n}$ and $Y$, we write $X_{1}(\mathbb{K}) \cdots X_{n}(\mathbb{K}) \hookrightarrow Y(\mathbb{K})$ if $(\lambda_{j}^{1} \cdots \lambda_{j}^{n})_{j=1}^{\infty} \in Y(\mathbb{K})$ and
$$ \left\| (\lambda_{j}^{1} \cdots \lambda_{j}^{n})_{j=1}^{\infty} \right\|_{Y(\mathbb{K})} \leq \prod\limits_{m=1}^{n} \left\|(\lambda_{j}^{m})_{j=1}^{\infty}\right\|_{X(\mathbb{K})}$$
whenever $(\lambda_{j}^{m})_{j=1}^{\infty} \in X_{m}(\mathbb{K}), ~m=1,\ldots,n$.

In the case $X_1 = \cdots = X_n = X$ we simply write $X(\mathbb{K})^n \hookrightarrow Y(\mathbb{K})$.

\section{Hyper-ideals of multilinear operators}\label{section 2}

In this section we show how sequence classes can produce Banach hyper-ideals. %Just one condition more from \cite{jamilson} is needed:

\begin{definition}\rm \cite{jamilsonmonat} The sequence classe $X$ is said to be {\it multilinearly stable} if for all $n\in \mathbb{N}$,  Banach spaces $E_1, \ldots, E_n,F$  and  $A \in \mathcal{L}(E_{1},\ldots,E_{n};F)$, $(A(x_{j}^{1},\ldots,x_{j}^{n}))_{j=1}^{\infty} \in  X(F)$ whenever $(x_{j}^{m})_{j=1}^{\infty} \in X(E_{m})$, $m=1,\ldots,n$ and, in this case, the  $n$-linear induced operator
$$\widetilde{A} \colon X(E_{1}) \times \cdots \times X(E_{n}) \longrightarrow X(F)~,~A\left((x_{j}^{1})_{j=1}^\infty,\ldots,(x_{j}^{n})_{j=1}^\infty \right) =  (A(x_{j}^{1},\ldots,x_{j}^{n}))_{j=1}^{\infty},$$
is continuous with $\|\widetilde{A}\| = \|A\|$.
\end{definition}

\begin{example}\label{exmultest}\rm For a Banach space $E$, consider
$${\rm RAD}(E) : = \left\{(x_{j})_{j=1}^{\infty} \in E^{\mathbb{N}} : \|(x_{j})_{j=1}^{\infty}\|_{{\rm RAD}(E)} = \sup\limits_{k} \|(x_{j})_{j=1}^{k}\|_{{\rm Rad}(E)} < \infty\right\}$$
(see \cite{djt, tarieladze}) and the closed subspace $\ell_p^u(E)$ of $\ell_p^w(E)$, $1 \leq p < \infty$, formed by unconditionally $p$-summable sequences, that is
$$\ell_p^u(E) = \left\{(x_j)_{j=1}^\infty \in \ell_p^w(E) : \lim\limits_k \|(x_j)_{j=k}^\infty\|_{w,p} = 0\right\} $$
(see  \cite{df}). The following sequence classes are multilinearly stable: $\ell_{\infty}(\cdot)$, $c_{0}(\cdot)$, ${\rm Rad}(\cdot)$, ${\rm RAD}(\cdot)$, $\ell_1^w(\cdot)$, $\ell_1^u(\cdot)$, $\ell_{p}(\cdot)$, $\ell_{p}\langle \,\cdot\, \rangle$, $1 \leq p < + \infty$ (see \cite{jamilsonmonat}).
\end{example}

Given $n \in \mathbb{N}$, sequence classes $X,Y$ and Banach spaces $E_1, \ldots, E_n,F$, an operator $A \in {\cal L}(E_1, \ldots, E_n;F)$ is said to be {\it $(X;Y)$-summing} if  $(A(x_{j}^{1},\ldots,x_{j}^{n}))_{j=1}^{\infty} \in  Y(F)$ whenever $(x_{j}^{m})_{j=1}^{\infty} \in X(E_{m})$, $m=1,\ldots,n$. In this case we write $A \in {\cal L}_{X_1, \ldots, X_n;Y}(E_1, \ldots, E_n;F)$ and, according to \cite[Proposition 2.4]{jamilsonmonat}, the $n$-linear induced operator
$$\widetilde{A} \colon X(E_{1}) \times \cdots \times X(E_{n}) \longrightarrow Y(F)~,~A\left((x_{j}^{1})_{j=1}^\infty,\ldots,(x_{j}^{n})_{j=1}^\infty \right) =  (A(x_{j}^{1},\ldots,x_{j}^{n}))_{j=1}^{\infty},$$
is continuous. We define $\|A\|_{X;Y}=\|\widetilde{A}\|$.

\begin{theorem} \label{L_XYhiperideal}
Let $X$ and $Y$ be sequence classes with $X$ multilinearly stable, $Y$  linearly stable and $X(\mathbb{K})^n \hookrightarrow Y(\mathbb{K})$ for every $n \in \mathbb{N}$. Then $(\mathcal{L}_{X;Y},\|\cdot\|_{X;Y})$ is a Banach hyper-ideal of multilinear operators.
\end{theorem}

\begin{proof} By \cite[Theorem 3.6]{jamilsonmonat} we know that ${\cal L}_{X;Y}$ is a Banach multi-ideal, so the hyper-ideal property is all that is left to be proved.  To do so, let $B_{1} \in \mathcal{L}(G_{1},\ldots,G_{m_{1}};E_{1}), \ldots, B_{n} \in \mathcal{L}(G_{m_{n-1}+1},\ldots,G_{m_{n}};E_{n})$, $t \in \mathcal{L}(F;H)$ and $A \in \mathcal{L}_{X;Y}(E_{1},\ldots,E_{n};F)$. Given sequences $(x_{j}^{k})_{j=1}^{\infty} \in X(G_{k})$, $k=1,\ldots,m_{n}$, the multilinear stability of $X$ gives
$$\left(B_1(x_{j}^{1}, \ldots, x_{j}^{m_1}) \right)_{j=1}^\infty \in X(E_1), \ldots,  \left(B_n(x_{j}^{m_{n-1}+1}, \ldots, x_{j}^{m_n}) \right)_{j=1}^\infty \in X(E_n) $$
and
$$\|\widetilde{B_1} \colon X(G_1) \times \cdots \times X(G_{m_1})\longrightarrow E_1\|= \|B_1\|, \ldots, $$
$$\|\widetilde{B_n} \colon X(G_{m_{n-1}+1}) \times \cdots \times X(G_{m_n})\longrightarrow E_n\|= \|B_n\|. $$
Since $A$ is $(X;Y)$-summing,
$$\left(A\left(B_1(x_{j}^{1}, \ldots, x_{j}^{m_1}), \ldots, B_n(x_{j}^{m_{n-1}+1}, \ldots, x_{j}^{m_n}) \right)\right)_{j=1}^\infty \in Y(F)$$
and
$$\|\widetilde{A} \colon X(E_1)\times \cdots X(E_n) \longrightarrow Y(F) \| = \|A\|_{X;Y}. $$
Finally the linear stability of $Y$  yields $\|\widetilde{t} \colon Y(F) \longrightarrow Y(H) = \|t\|$ and
$$\left((t \circ A \circ (B_1, \ldots, B_n)\left( x_j^1, \ldots, x_j^{m_n}\right) \right)_{j=1}^\infty=$$
$$= \left(t\left(A\left(B_1(x_{j}^{1}, \ldots, x_{j}^{m_1}), \ldots, B_n(x_{j}^{m_{n-1}+1}, \ldots, x_{j}^{m_n}) \right)\right)\right)_{j=1}^\infty \in Y(H). $$
Therefore,
%$$\left\|\left(B_1(x_{j}^{1}, \ldots, x_{j}^{m_1}) \right)_{j=1}^\infty\right\|_{X(E_1)} \leq \|B_1\| \cdot  $$
% $((B_{1},\ldots,B_{n})(x_{j}^{1},\ldots,x_{j}^{m_{n}}))_{j=1}^{\infty} \in X(E_{m})$ e $((A\circ(B_{1},\ldots,B_{n}))(x_{j}^{1},\ldots,x_{j}^{m_{n}}))_{j=1}^{\infty} \in Y(F)$. Como $Y$ é linearmente estável,
%$$((t\circ A \circ (B_{1},\ldots,B_{n}))(x_{j}^{1},\ldots,x_{j}^{m_{n}}))_{j=1}^{\infty} \in Y(H).$$
%Logo,
$(t \circ A \circ (B_{1},\ldots,B_{n})) \in \mathcal{L}_{X;Y}(G_1, \ldots, G_{m_n};H)$ and  % e, portanto, $\mathcal{L}_{X;Y}$ é um hiper-ideal de Banach. Além disso,
%E da igualdade acima também segue que
%$$[t \circ A \circ (B_{1},\ldots,B_{n})]^{\sim} = \widetilde{t} \circ \widetilde{A} \circ (\widetilde{B_1}, \ldots, \widetilde{B_n}). $$
%Daí,
\begin{align*} \|t \circ A \circ (B_{1},\ldots,B_{n})\|_{X;Y} & = \|[t \circ A \circ (B_{1},\ldots,B_{n})]^{\sim}\| = \| \widetilde{t} \circ \widetilde{A} \circ (\widetilde{B_1}, \ldots, \widetilde{B_n})\| \\
& \leq \| \widetilde{t}\|\cdot \| \widetilde{A} \|\cdot \| \widetilde{B_1}\| \cdots  \|\widetilde{B_n}\| = \|t\| \cdot \|A\|_{X;Y} \cdot \|B_{1}\| \cdots \|B_{n}\|.
\end{align*}
%o que completa a demonstração.
\end{proof}

Before giving concrete examples, let us see a property of $\mathcal{L}_{X;Y}$ thas has gone unnoticed and will be helpful later.

By $S_n$ we denote the set of permutations of the set $\{1, \ldots, n\}$. According to Floret and Garc\'ia \cite{fg}, a normed subclass $(\mathcal{G},\|\cdot\|_{\mathcal{G}})$ of the class of multilinear operators between Banach spaces is {\it strongly symmetric} if for all Banach spaces $E$ and $F$, $n \in \mathbb{N}$, a permutation $\sigma \in S_n$ and $A\in \mathcal{G}(^{n}E;F)$, it holds $A_\sigma \in \mathcal{G}(^{n}E;F)$  and $\|A_\sigma\|_{\mathcal{G}} = \|A\|_{\mathcal{G}}$, where
$$A_\sigma(x_1, \ldots, x_n) = A(x_{\sigma(1)}, \ldots, x_{\sigma(n)}).$$

\begin{proposition} \label{L_XYfortSimetrico}
For all sequence classes $X$ and $Y$,  the normed class $(\mathcal{L}_{X;Y},\|\cdot\|_{X;Y})$  of multilinear operators is strongly symmetric.
\end{proposition}

\begin{proof} Given $T \in {\cal L}(^nG;H)$ and $\sigma \in S_n$, from the obvious equality
$$\{\|T(z_{\sigma(1)}, \ldots, z_{\sigma(n)})\|: z_1, \ldots, z_n \in B_G\} = \{\|T(z_1, \ldots, z_n)\|: z_1, \ldots, z_n \in B_G\} $$
it follows easily that $\|T\| = \|T_\sigma\|$.

Let $n \in \mathbb{N}$, Banach spaces $E,F$, $\sigma \in S_{n}$ and $A \in \mathcal{L}_{X;Y}(^{n}E;F)$ be given. Then the induced operator $\widetilde{A} \colon X(E)^{n} \longrightarrow Y(F)$ %, ~\widetilde{A}((x_{j}^{1})_{j=1}^{\infty},\ldots,(x_{j}^{n})_{j=1}^{\infty}) = (A(x_{j}^{1}, \ldots, x_{j}^{n}))_{j=1}^{\infty}, $$
is well defined, $n$-linear and continuous, so is its permutation %. Podemos então considerar a permutação de $\widetilde{A}$ por $\sigma$,
$ (\widetilde{A})_{\sigma} \colon X(E)^{n} \longrightarrow Y(F).$
%o qual é igualmente $n$-linear e contínuo.
We can also consider the operator induced by $A_{\sigma}$, namely $ \widetilde{A_{\sigma}} \colon X(E)^{n} \longrightarrow F^{\mathbb{N}}.$ A standard computation shows that
\begin{equation}\label{noviguald}
(\widetilde{A} )_\sigma = \widetilde{A_\sigma},
\end{equation}
%Com efeito, para todas sequências  $(x_{j}^{m})_{j=1}^{\infty} \in X(E)$, $m=1,\ldots,n$,
%\begin{align*} \left(\widetilde{A} \right)_\sigma \left(\left(x_j^1\right)_{j=1}^\infty, \ldots, \left(x_j^n\right)_{j=1}^\infty \right) &= \widetilde{A} \left(\left(x_j^{\sigma(1)}\right)_{j=1}^\infty, \ldots, \left(x_j^{\sigma(n)}\right)_{j=1}^\infty \right)\\
%& = \left(A\left(x_j^{\sigma(1)}, \ldots, x_j^{\sigma(n)}\right) \right)_{j=1}^\infty\\
%& = \left(A_\sigma(x_j^{1}, \ldots, x_j^{n}) \right)_{j=1}^\infty\\
%& = \widetilde{A_\sigma} \left((x_j^1)_{j=1}^\infty, \ldots, (x_j^n)_{j=1}^\infty \right).
%\end{align*}
%Como $\left(\widetilde{A}\right)_\sigma$ toma valores em $Y(F)$, segue de (\ref{noviguald}) que
from which it follows that $\widetilde{A_\sigma}$ takes values in $Y(F)$, proving that $A_\sigma \in \mathcal{L}_{X;Y}(^{n}E;F)$. Furthermore,
$$\|A\|_{X;Y} = \|\widetilde{A}\|_{{\cal L}(^nX(E);Y(F))} =  \|(\widetilde{A})_\sigma\|_{{\cal L}(^nX(E);Y(F))} = \|\widetilde{A_\sigma}\|_{{\cal L}(^nX(E);Y(F))} = \|A_\sigma\|_{X;Y} , $$
where the second equality follows from the begining of the proof and the third from (\ref{noviguald}).
\end{proof}

      Now we apply Theorem \ref{L_XYhiperideal} to show that several well studied Banach multi-ideals are actually Banach hyper-ideals. We shall use the obvious fact that every multilinearly stable sequence class is linearly stable. %permite concluir que vários multi-ideais já estudados são na verdade hiper-ideais de Banach de operadores multilineares.

\begin{example}\rm Absolutely summing multilinear operators were introduced in   \cite{alencarmatos} and then extensively studied by a number of authors (see, for instance, the references in \cite{davidson}). By definition, an operator $A \in {\cal L}(E_1, \ldots, E_n;F)$ is absolutely summing if  $(A(x_{j}^{1},\ldots,x_{j}^{n}))_{j=1}^{\infty} \in  \ell_1(F)$ whenever $(x_{j}^{m})_{j=1}^{\infty} \in \ell_1^w(E_{m})$, $m=1,\ldots,n$. Since the sequence classes $\ell_1(\cdot)$ and $\ell_{1}^{w}(\cdot)$ are multilinearly stable (Example \ref{exmultest}), Theorem \ref{L_XYhiperideal} assures that the multi-ideal of absolutely summing operators is actually a Banach hyper-ideal. %que o multi-ideal dos operadores absolutamente somantes é um hiper-ideal de Banach.
\end{example}

\begin{example}\rm Weakly summing multilinear operators were studied in  \cite{jamilsonmonat, kim, popaillinois, tesecarlao} among others. A well known characterization says that an operator $A \in {\cal L}(E_1, \ldots, E_n;F)$ is weakly summing if and only if $(A(x_{j}^{1},\ldots,x_{j}^{n}))_{j=1}^{\infty} \in  \ell_1^w(F)$ whenever $(x_{j}^{m})_{j=1}^{\infty} \in \ell_1^w(E_{m})$, $m=1,\ldots,n$. Since the sequence class $\ell_{1}^{w}(\cdot)$ is multilinearly stable (Example \ref{exmultest}), Theorem \ref{L_XYhiperideal} gives that the multi-ideal of weakly summing operators is actually a Banach hyper-ideal.%Na nossa notação, trata-se da classe ${\cal L}_{\ell_{1}^{w}(\cdot);\ell_{1}^{w}(\cdot)}$. Vimos no Exemplo \ref{exmultest} que a classe de sequências $\ell_{1}^{w}(\cdot)$ é multilinearmente estável. Segue então do Teorema \ref{L_XYhiperideal} que o multi-ideal dos operadores fracamente somantes é um hiper-ideal de Banach.
\end{example}

\begin{example}\rm Multilinear operators of type $p$, $1 \leq p \leq 2$, and of cotype $q$, $2 \leq q < +\infty$,  were investigated in \cite{jamilsoncomplut, davidson}. Let $A \in {\cal L}(E_1, \ldots, E_n;F)$. As proved in  \cite{jamilsoncomplut}, $A$ has type   $p$ if and only if $(A(x_{j}^{1},\ldots,x_{j}^{n}))_{j=1}^{\infty} \in  {\rm Rad}(F)$ whenever $(x_{j}^{m})_{j=1}^{\infty} \in \ell_p(E_{m})$, $m=1,\ldots,n$; and $A$ has cotype $q$ if and only if $(A(x_{j}^{1},\ldots,x_{j}^{n}))_{j=1}^{\infty} \in  \ell_q(F)$ whenever $(x_{j}^{m})_{j=1}^{\infty} \in {\rm Rad}(E_{m})$, $m=1,\ldots,n$. %Na nossa notação, trata-se da classe ${\cal L}_{\ell_p(\cdot);{\rm Rad}(\cdot)}$.
Since the sequence classes $\ell_p(\cdot), \ell_q(\cdot)$ and ${\rm Rad}(\cdot)$ are multilinearly stable (Example \ref{exmultest}), from Theorem \ref{L_XYhiperideal} we conclude that the multi-ideals of operators of type $p$ and of operators of cotype $q$ are actually Banach hyper-ideals.
 \end{example}

\begin{example}\rm Cohen almost summing multilinear operators were introduced in  \cite{bushiJMAA2013}. An operator $A \in {\cal L}(E_1, \ldots, E_n;F)$ is Cohen almost summing  if and only if $(A(x_{j}^{1},\ldots,x_{j}^{n}))_{j=1}^{\infty} \in  \ell_2\langle F \rangle$ whenever $(x_{j}^{m})_{j=1}^{\infty} \in {\rm Rad}(F)$, $m=1,\ldots,n$. Since the sequence classes ${\rm Rad}(\cdot)$ and $\ell_2\langle\cdot \rangle$ are multilinearly stable (Example \ref{exmultest}), from Theorem \ref{L_XYhiperideal} it follows that the multi-ideal of Cohen almost summing operators is actually a Banach hyper-ideal.
\end{example}

Of course multilinearly stable and linearly stable sequence classes can be combined to produce Banach hyper-ideals that have not been studied yet. We finish this section with an illustrative example.

\begin{example}\rm \label{llexe}\rm Given $1 \leq p < \infty$, since the sequence classes $\ell_p(\cdot)$ and $\ell_p\langle \cdot \rangle$ are multilinearly stable, from Theorem \ref{L_XYhiperideal} we have that the class ${\cal L}_{\ell_p(\cdot); \ell_p\langle \cdot \rangle}$ is a Banach hyper-ideal. In the linear case, this class coincides with the Banach operator ideal of Cohen strongly $p$-summing operators, but in the multilinear case it does not coincide with the well studied class of Cohen strongly $p$-summing operators (see \cite{achourmezrag,  bushiJMAA2013, jamilsonlaa2013, jamilsonlama2014, mezragsaadi, mezragsaadi2, popa}). In fact, according to \cite{jamilsonlama2014}, an $n$-linear operator is Cohen strongly $p$-summing if and only if it sends sequences in $\ell_{np}(\cdot)$ to sequences in $\ell_p\langle \cdot \rangle$, and not sequences in $\ell_{p}(\cdot)$ to sequences in $\ell_p\langle \cdot \rangle$.
\end{example}

\section{Polynomial hyper and two-sided ideals} \label{section 3}

Classes of polynomials defined by the transformation of vector-valued sequences were treated in \cite{achourmediterranean}, but the connection with the corresponding classes of multilinear operators was not investigated there. We do it now. As usual, by $\widehat{A}$ we mean the $n$-homogeneous polynomial defined by the $n$-linear operator $A$ and by $\check{P}$ the symmetric $n$-linear operator associated to the $n$-homogeneous polynomial $P$.

\begin{theorem} \label{equivhippol} Let $X$ and $Y$ be sequence classes. The following are equivalent for a polynomial $P \in \mathcal{P}(^{n}E;F)$:\\
{\rm (a)} $\check{P} \in \mathcal{L}_{X;Y}(^{n}E;F)$.\\
{\rm (b)} There exists $A \in \mathcal{L}_{X;Y}(^{n}E;F)$ such that $\widehat{A} = P$.\\
{\rm (c)} $(P(x_{j}))_{j=1}^{\infty} \in Y(F)$ whenever $(x_{j})_{j=1}^{\infty} \in X(E)$.\\
{\rm (d)} The induced operator 
$$ \widetilde{P} \colon X(E) \longrightarrow Y(F)~,~\widetilde{P}((x_{j})_{j=1}^{\infty}) = (P(x_{j}))_{j=1}^{\infty},$$
is a well defined continuous $n$-homogeneous polynomial.

If  the conditions above hold, then
$$\|\widetilde{P}\| \leq \|\check{P}\|_{X;Y} = \inf\{\|A\|_{X;Y} : A \in {\cal L}_{X;Y} ~{\rm e~}\widehat{A} = P\} \leq \frac{n^n}{n!} \|\widetilde{P}\|.$$
\end{theorem} 

\begin{proof} Some of the equivalences can be found in the proof of \cite[Proposition 2.6]{achourmediterranean}. The remaining ones can be proved by standard arguments. We focus on the relationships between the norms. Let $A \in \mathcal{L}_{X;Y}(^{n}E;F)$ be such that $\widehat{A} = P$. For every sequence $(x_j)_{j=1}^\infty \in X(E)$,
$$\widehat{\LARGE{(}{\widetilde{A}}\LARGE{)}}\left((x_j)_{j=1}^\infty  \right) = {\widetilde{A}}\left((x_j)_{j=1}^\infty, \ldots, (x_j)_{j=1}^\infty  \right) = \left(A (x_j, \ldots, x_j)\right)_{j=1}^\infty = \left(P (x_j)\right)_{j=1}^\infty= \widetilde{P}\left( (x_j)_{j=1}^\infty\right),$$
that is, $\widetilde{P} = \widehat{\LARGE{(}{\widetilde{A}}\LARGE{)}}$. So, 
$$\|\widetilde{P}\| = \mbox{\large{$\|$}}\widehat{\LARGE{(}{\widetilde{A}}\LARGE{)}}\mbox{\large{$\|$}} \leq \|\widetilde{A}\| = \|A\|_{X;Y}. $$
Taking the infimum over all such $A$ and using that $\widehat{\check{P}} = P$ we get
%para toda $A \in \mathcal{L}_{X;Y}(^{n}E;F)$ tal que $\widehat{A} = P$. Segue que
$$\|\widetilde{P}\| \leq \inf\{\|A\|_{X;Y} : A \in {\cal L}_{X;Y} ~{\rm e~}\widehat{A} = P\} \leq \|\check{P}\|_{X;Y}. $$
%onde a última desigualdade vem de $(a)$.
Take again  $A \in \mathcal{L}_{X;Y}(^{n}E;F)$ such that $\widehat{A} = P$. Denoting by $A_s$ its symmetrization, that is, $A = \frac{1}{n!} \cdot \sum\limits_{\sigma \in S_n} A_\sigma$, we have $A_s = \check{P}$. By Proposition \ref{L_XYfortSimetrico} we know that ${\cal L}_{X;Y}$ is completely symmetric, hence $A_s \in {\cal L}_{X;Y}(^nE;F)$ and $\|A_\sigma\|_{X_;Y} = \|A\|_{X;Y}$ for every $\sigma \in S_n$. Therefore,
$$\|\check{P}\|_{X;Y} = \|A_s\|_{X;Y} = \left\|\frac{1}{n!}\cdot \sum\limits_{\sigma \in S_{n}} A_{\sigma}\right\|_{X;Y}\leq \frac{1}{n!}\cdot \sum\limits_{\sigma \in S_{n}} \|A_{\sigma}\|_{X;Y} = \|A\|_{X;Y}, $$
from which we conclude that $\inf\{\|A\|_{X;Y} : A \in {\cal L}_{X;Y} ~{\rm e~}\widehat{A} = P\} = \|\check{P}\|_{X;Y}$.

For the remaining inequality, note that we proved above that $\widetilde{P} = \widehat{\LARGE{(}{\widetilde{A}}\LARGE{)}}$ for every $A \in \mathcal{L}_{X;Y}(^{n}E;F)$ such that $\widehat{A} = P$. Since $\widehat{\check{P}} = P$  and $\check{P} \in \mathcal{L}_{X;Y}(^{n}E;F)$, we have  $\widetilde{P} = \widehat{\LARGE{(}\widetilde{\check{P}}\LARGE{)}}$. The symmetry of $\check{P}$ implies easily the symmetry of its induced $n$-linear operator $\widetilde{\check{P}}$, so \cite[Theorem 2.2]{mujica} gives
$$\|\check{P}\|_{X;Y}= \mbox{\large{$\|$}}\widetilde{\check{P}} \mbox{\large{$\|$}} \leq \frac{n^n}{n!} \cdot \mbox{\large{$\|$}}\widehat{\Large{(}\widetilde{\check{P}}\Large{)}} \mbox{\large{$\|$}} = \frac{n^n}{n!} \cdot \|\widetilde{P}\|. $$
\end{proof}

Now we can define $(X;Y)$-summing polynomials.

\begin{definition}\label{deffed}\rm
 Given sequence classes $X$ and $Y$, we say that a polynomial $P \in \mathcal{P}(^{n}E;F)$ is {\it $(X;Y)$-summing}, in symbols $P \in \mathcal{P}_{X;Y}(^{n}E;F)$, if  the equivalent conditions of the previous theorem hold for $P$. In this case we define
$$\|P\|_{X;Y;1} = \|\widetilde{P}\| \mbox{~~and~~} \|P\|_{X;Y;2} = \|\check{P}\|_{X;Y}. $$
\end{definition}

\begin{theorem}\label{P_XYhiperideal}
Let $X$ be a multilinearly stable sequence class and $Y$ be a linearly stable sequence class such that $X(\mathbb{K})^n \hookrightarrow Y(\mathbb{K})$ for every $n \in \mathbb{N}$. Then $(\mathcal{P}_{X;Y}, \|\cdot\|_{X;Y;1})$ and $(\mathcal{P}_{X;Y}, \|\cdot\|_{X;Y;2})$ are $\left( \frac{n^n}{n!}\right)_{n=1}^\infty$-polynomial Banach hyper-ideals.
\end{theorem}

\begin{proof} Using that $(\mathcal{L}_{X;Y}, \|\cdot\|_{X;Y})$ is a Banach hyper-ideal (Theorem \ref{L_XYhiperideal}) and that a polynomial $P$ belongs to $\mathcal{P}_{X;Y}$ if and only if there is $A$ in  $\mathcal{L}_{X;Y}$ such that $\widehat{A} = P$ (Theorem \ref{equivhippol}), it follows from \cite[Theorem 4.3]{ewertonarxiv} that $(\mathcal{P}_{X;Y}, \|\cdot\|_{X;Y;2})$ is a $\left( \frac{n^n}{n!}\right)_{n=1}^\infty$-polynomial Banach hyper-ideal. The issues related to the norm $\|\cdot\|_{X;Y;1}$ are all that is left to be checked. The norm axiomas are straightforward. By Theorem \ref{equivhippol} the norms $\|\cdot\|_{X;Y;1}$ and $\|\cdot\|_{X;Y;2}$ are equivalent on each component $\mathcal{P}_{X;Y}(^nE;F)$. Since this space is complete with respect to $\|\cdot\|_{X;Y;2}$, it is also complete with respect to $\|\cdot\|_{X;Y;1}$. Let us prove that $\|\hat{I}_{n} \colon \mathbb{K} \longrightarrow \mathbb{K}, ~\hat{I}_{n}(\lambda) = \lambda^{n}\|_{X;Y;1} = 1$ for every $n$. Note that
$$ \|\hat{I}_{n}\|_{X;Y;1} = \mbox{\large{$\|$}}\widetilde{\hat{I}_{n}} \mbox{\large{$\|$}} %\sup\left\{\left\|\widetilde{\hat{I}_{n}}\left((\lambda_j)_{j=1}^\infty \right) \right\|_{Y(\mathbb{K})} : (\lambda_j)_{j=1}^\infty \in B_{X(\mathbb{K})} \right\}\\
%& =\sup\left\{\left\|\left({\hat{I}_{n}}(\lambda_j)\right)_{j=1}^\infty \right\|_{Y(\mathbb{K})} : (\lambda_j)_{j=1}^\infty \in B_{X(\mathbb{K})} \right\}\\
= \sup\left\{\mbox{\large{$\|$}}\left(\lambda_j^n\right)_{j=1}^\infty \mbox{\large{$\|$}}_{Y(\mathbb{K})} : (\lambda_j)_{j=1}^\infty \in B_{X(\mathbb{K})} \right\}.
$$
On the one hand, taking $e_1 = (1,0,0,0\ldots,)$ we get $\|e_1\|_{X(\mathbb{K})} = \|e_1\|_{Y(\mathbb{K})} =1 $. The expression above gives $\|\hat{I}_{n}\|_{X;Y;1} \geq 1$. On the other hand, combining the experession above with $X(\mathbb{K})^n \hookrightarrow Y(\mathbb{K})$ we obtain
$$\|\hat{I}_{n}\|_{X;Y;1} \leq \sup\left\{\mbox{\large{$\|$}}\left(\lambda_j\right)_{j=1}^\infty \mbox{\large{$\|$}}^n_{Y(\mathbb{K})} : (\lambda_j)_{j=1}^\infty \in B_{X(\mathbb{K})} \right\} \leq 1. $$

To check the hyper-ideal inequality, let $P \in {\cal P}_{X;Y}(^{n}E;F)$, $Q \in \mathcal{P}(^{m}G;E)$ and $t \in \mathcal{L}(F;H)$ be given. We already know that $t \circ P \circ Q \in {\cal P}_{X;Y}(^{mn}G;H)$. %Veja que para cada sequência $(x_j)_{j=1}^\infty \in X(E)$,
The fact that $\widetilde{Q} \colon X(G) \longrightarrow X(E)$ is the polynomial associated to the symmetric operator $\widetilde{\check{Q}} \colon X(G)^m \longrightarrow X(E)$ justifies the first inequality below:
$$\|\widetilde{Q}\| \leq \mbox{\large{$\|$}}\widetilde{\check{Q}} \mbox{\large{$\|$}} = \|\check{Q}\| \leq \frac{m^m}{m!}\|Q\|, $$
where the equality follows from the multilinear stability of $X$. Therefore, %De forma análoga ao que foi feito na demonstração do Teorema \ref{L_XYhiperideal}, prova-se que $[t \circ P \circ Q]^{\sim} = \widetilde{t} \circ \widetilde{P} \circ \widetilde{Q}$. Daí,
\begin{align*}\|t \circ P \circ Q\|_{X;Y;1}& = \|[t \circ P \circ Q]^{\sim}\| = \|\widetilde{t} \circ \widetilde{P} \circ \widetilde{Q}\| \leq \|\widetilde{t}\| \cdot \|\widetilde{P}\| \cdot \|\widetilde{Q}\|^m\\
& = \|t\| \cdot \|P\|_{X;Y;1} \cdot \|\widetilde{Q}\|^m \leq \left(\frac{m^m}{m!}\right)^m  \|t\| \cdot \|P\|_{X;Y;1} \cdot\|Q\|^m.
\end{align*}
\end{proof}

The proof above makes clear which condition should be imposed on the sequence class $X$ for $(\mathcal{P}_{X;Y}, \|\cdot\|_{X;Y;1})$ to be a polynomial Banach hyper-ideal without constants (that is, $C_n = 1$ for every $n$).

\begin{definition}\rm A sequence class $X$ is {\it polinomially stable} if, regardless of the $n \in \mathbb{N}$, the Banach spaces $E$ and $F$  and $P \in {\cal P}(^nE;F)$, $(P(x_j))_{j=1}^\infty \in X(F)$ whenever $(x_j)_{j=1}^\infty \in X(E)$ and $\widetilde{P} \colon X(E) \longrightarrow X(F)$ is an $n$-homogeneous polynomial with $\|\widetilde{P}\| = \|P\|$.
\end{definition}

\begin{example}\label{expolest}\rm It is not difficult to see that the sequence classes $\ell_p(\cdot)$, $1 \leq p < \infty$, are polynomially stable.
\end{example}

Bearing in mind Theorem \ref{P_XYhiperideal} and the final part of its proof,  a moment's reflection gives the following result.

\begin{theorem}\label{teooet} Let $X$ be a polynomially stable sequence class and  $Y$ be a linearly stable sequence class such that $X(\mathbb{K})^n \hookrightarrow Y(\mathbb{K})$ for every $n \in \mathbb{N}$. Then $(\mathcal{P}_{X;Y}, \|\cdot\|_{X;Y;1})$ is a polynomial Banach hyper-ideal. %é um hiper-ideal de Banach de polinômios homogêneos.
\end{theorem}

Next we proceed to the construction of polynomial two-sided ideals. 

\begin{theorem}\label{teoteo}
Let $X$ and $Y$ be multilinearly stable sequence classes such that $X(\mathbb{K})^n \hookrightarrow Y(\mathbb{K})$ for every $n \in \mathbb{N}$. Then $(\mathcal{P}_{X;Y}, \|\cdot\|_{X;Y;2})$ is a $\left(\frac{n^n}{n!},\frac{n^n}{n!} \right)$-polynomial Banach two-sided ideal.
\end{theorem}

\begin{proof} Since $({\cal L}_{X;Y}; \|\cdot\|_{X;Y})$ is a Banach hyper-ideal of multilinear operataors (Theorem \ref{L_XYhiperideal}), by Theorem \ref{equivhippol} it is enough to show that this hyper-ideal fulfills the assumption of \cite[Proposition 4.2]{ewertonjmaa}. To do so, let $A \in {\cal L}_{X;Y}(E_1, \ldots, E_n;G)$, $m,r \in \mathbb{N}$ be such that $mr = n$, $G$ be a Banach space, $B_1 \in {\cal L}_{X;Y}(E_1, \ldots, E_m;G), \ldots,$ $B_r \in {\cal L}(E_{(r-1)m+1}, \ldots, E_n;G)$ and $C \in {\cal L}(^rG;G)$ be a symmetric operator such that $A = C \circ (B_1, \ldots, B_r)$. For all $(x_j^k)_{j=1}^\infty \in X(E_k)$, $k = 1, \ldots, n$,  $$(B_1(x_j^1, \ldots, x_j^m))_{j=1}^\infty, \ldots, (B_r(x_j^{(r-1)m+1}, \ldots, x_j^n))_{j=1}^\infty \in Y(G) ~{\rm and}$$
$$\|\widetilde{B_1}\colon X(E_1) \times \cdots \times X(E_m)\longrightarrow Y(G)\| = \|B_1\|, \ldots, $$
$$\|\widetilde{B_r}\colon X(E_{(r-1)m+1}) \times \cdots \times X(E_n)\longrightarrow G\| = \|B_r\|. $$
The multilinear stability of $Y$ gives
\begin{align*}(A(x_j^1, \ldots, x_j^n))_{j=1}^\infty &= ((C\circ (B_1, \ldots, B_r)(x_j^1, \ldots, x_j^n))_{j=1}^\infty\\
& =\left(C(B_1(x_j^1, \ldots, x_j^m), \ldots, B_r(x_j^{(r-1)m+1}, \ldots, x_j^n)) \right)_{j=1}^\infty \in Y(F),
\end{align*}
showing that $A \in {\cal L}_{X;Y}(E_1, \ldots, E_n;F)$ and, moreover,
\begin{align*}\|A\|_{X;Y}&=\|\widetilde{A} \colon X(E_1) \times \cdots, X(E_n) \longrightarrow Y(F)\| =  \| \widetilde{C} \circ (\widetilde{B_1}, \ldots, \widetilde{B_r}\|\\ & \leq \|\widetilde{C} \colon Y(G)^r \longrightarrow Y(F)\|\cdots \|\widetilde{B_1}\|\cdots \|\widetilde{B_r}\| = \|C\|\cdot \|B_1\| \cdots \|B_r\|.
\end{align*}
\end{proof}

With polynomial stability we can go a bit further. 

\begin{theorem} \label{teobilat3}
Let $X$ and $Y$ be sequence classes such that $X$ is polynomially stable and $X(\mathbb{K})^n \hookrightarrow Y(\mathbb{K})$ for every $n \in \mathbb{N}$.\\
{\rm (a)} If $Y$ is multilinearly stable, then $(\mathcal{P}_{X;Y}, \|\cdot\|_{X;Y;1})$ is a  $\left(1, \frac{n^n}{n!}\right)_{n=1}^\infty$-polynomial Banach two-sided ideal.\\
{\rm (b)} If $Y$ is polynomially stable, then $(\mathcal{P}_{X;Y}, \|\cdot\|_{X;Y;1})$ is a Banach polynomial two-sided ideal. 
\end{theorem}

\begin{proof} We just prove (a). By the previous theorem it is enough to show that the two-sided ideal inequality holds with constants $\left(1, \frac{n^n}{n!}\right)_{n=1}^\infty$. Let $m,n, r \in \mathbb{N}$, $P \in \mathcal{P}_{X;Y}(^{n}E;F)$, $Q \in \mathcal{P}(^{m}G;E)$ and $R \in \mathcal{P}(^{r}F;H)$ be given. %By Theorem \ref{teooet} we know that $\|P\circ Q\|_{X;Y;1} \leq \|P\|_{X;Y;1}\cdot \|Q\|^n$. 
Denoting $\widetilde{R}\colon Y(F) \longrightarrow Y(H)$, $[P \circ Q]^\sim \colon X(G) \longrightarrow Y(F)$ and $[R\circ P \circ Q]^\sim \colon X(G) \longrightarrow Y(H)$, it holds
$[R\circ P \circ Q]^\sim = \widetilde{R} \circ  [P \circ Q]^\sim.$ As in the proof of Theorem \ref{equivhippol}, $\widetilde{R} = \widehat{\LARGE{(}\widetilde{\check{R}}\LARGE{)}}$, thus the multilinear stability of $Y$ gives 
%\begin{align*} \|R\circ P \circ Q\|_{X;Y;1} = \|[R\circ P \circ Q]^\sim\| \leq \|R\|\cdot
%\end{align*}
\begin{align*}\|R\circ P \circ Q\|_{X;Y;1} &=  \|\widetilde{R} \circ [P\circ Q] ^\sim\| \leq \|\widetilde{R}\|\cdot \|P \circ Q\|_{X;Y;1}^r \leq \|\widetilde{R}\|\cdot \|P\|_{X;Y;1}^r\cdot \|Q\|^{rn}\\
&= \mbox{\large{$\|$}}\widehat{\LARGE{(}\widetilde{\check{R}}\LARGE{)}}\mbox{\large{$\|$}}\cdot \|P\|_{X;Y;1}^r\cdot \|Q\|^{rn} \leq \mbox{\large{$\|$}}\widetilde{\check{R}}\mbox{\large{$\|$}}\cdot \|P\|_{X;Y;1}^r\cdot \|Q\|^{rn} \\
& = \left\|\check{R}\right\|\cdot \|P\|_{X;Y;1}^r\cdot \|Q\|^{rn} \leq \frac{r^r}{r!}\cdot \left\|R\right\|\cdot \|P\|_{X;Y;1}^r\cdot \|Q\|^{rn},
\end{align*}
where the second inequality follows from Theorem \ref{teooet}.
\end{proof}

As we did in the multilinear case, we shall show that some well studied polynomial ideal are actually hyper or two-sided ideals. The last example concerns a new class. To avoid unnecessary repetitions, we shall be quite briefer.

\begin{example}\rm Applying the theorems proved in this section for the sequence classes we saw in Examples \ref{exmultest} and \ref{expolest}, we get the following examples:%serem linearmente estáveis, multilinearmente estáveis e polinomialmente estáveis, obtemos os seguintes exemplos:

\medskip

\noindent (a) The class ${\cal P}_{\ell_1^w(\cdot); \ell_1(\cdot)}$ of absolutely summing polynomials and the class ${\cal P}_{\ell_1^w(\cdot); \ell_1^w(\cdot)}$ of weakly summing polynomials are $\left( \frac{n^n}{n!}\right)_{n=1}^\infty$-polynomial Banach hyper-ideals with the norm $\|\cdot\|_{X;Y;1}$ and $\left( \frac{n^n}{n!}, \frac{n^n}{n!}\right)_{n=1}^\infty$-polynomial Banach two-sided ideals with the norm $\|\cdot\|_{X;Y;2}$.

%\medskip
%
%\noindent (b) A classe ${\cal P}_{\ell_1^w(\cdot); \ell_1^w(\cdot)}$ dos polinômios homogêneos fracamente somantes é um $\left( \frac{n^n}{n!}\right)_{n=1}^\infty$-hiper-ideal de Banach de polinômios com a norma $\|\cdot\|_{X;Y;1}$ e um $\left( \frac{n^n}{n!}, \frac{n^n}{n!}\right)_{n=1}^\infty$-ideal bilateral de Banach de polinômios com a norma $\|\cdot\|_{X;Y;2}$.

\medskip

\noindent (b) For $1 \leq p \leq 2$, the class ${\cal P}_{\ell_p(\cdot); {\rm Rad}(\cdot)}$ of polynomials of type $p$ is a $\left( 1, \frac{n^n}{n!}\right)_{n=1}^\infty$-polynomial Banach two-sided ideal with the norm $\|\cdot\|_{X;Y;1}$ and a $\left( \frac{n^n}{n!}, \frac{n^n}{n!}\right)_{n=1}^\infty$-polynomial Banach two-sided ideal with the norm $\|\cdot\|_{X;Y;2}$.

\medskip

\noindent (c) For $2 \leq q < +\infty$, the class ${\cal P}_{{\rm Rad}(\cdot); \ell_q(\cdot)}$ of polynomials of cotype $q$ is a $\left( \frac{n^n}{n!}\right)_{n=1}^\infty$-polynomial Banach hyper-ideal with the norm $\|\cdot\|_{X;Y;1}$ and a $\left( \frac{n^n}{n!}, \frac{n^n}{n!}\right)_{n=1}^\infty$-polynomial Banach two-sided ideal with the norm $\|\cdot\|_{X;Y;2}$.

\medskip

\noindent (d) The class ${\cal P}_{{\rm Rad}(\cdot); \ell_2\langle\cdot\rangle}$ of Cohen almost summing polynomials is a $\left( \frac{n^n}{n!}\right)_{n=1}^\infty$-polynomial Banach hyper-ideal with the norm $\|\cdot\|_{X;Y;1}$ and a $\left( \frac{n^n}{n!}, \frac{n^n}{n!}\right)_{n=1}^\infty$-polynomial Banach two-sided ideal with the  norm $\|\cdot\|_{X;Y;2}$.

\medskip

\noindent (e) For $1 \leq p < +\infty$, the class ${\cal P}_{\ell_p(\cdot); \ell_p\langle \cdot\rangle}$ of polynomials associated to the multi-ideal ${\cal L}_{\ell_p(\cdot); \ell_p\langle \cdot\rangle}$ from Example \ref{llexe} is a $\left( 1, \frac{n^n}{n!}\right)_{n=1}^\infty$-polynomial Banach two-sided ideal with the norm $\|\cdot\|_{X;Y;1}$ and a $\left( \frac{n^n}{n!}, \frac{n^n}{n!}\right)_{n=1}^\infty$-polynomial Banach two-sided ideal with the norm $\|\cdot\|_{X;Y;2}$.
\end{example}

\section{Composition ideals}\label{section 4}

Composition multi-ideals are always hyper-ideals \cite{ewertonlaa}, composition polynomial ideals  are always hyper-ideals \cite{ewertonarxiv} and sometimes they are two-sided ideals \cite{ewertonjmaa}. To make sure that we are creating new hyper and two-sided ideals in this paper, in this section we check that the multi-ideals ${\cal L}_{X;Y}$ and the polynomial ideals ${\cal P}_{X;Y}$ are not composition ideals in general. Our main results hold for $X$ multilinearly stable, so we have to give a counterexample with $X$ multilinearly stable. 

The notion of composition ideals goes back to Pietsch \cite{pietsch83}, now we give the definitions that matter in our environment. Let $X$ and $Y$ be sequence classes such that $({\cal L}_{X;Y},\|\cdot\|_{X;Y})$ is a Banach multi-ideal, that is, $X$ and $Y$ are linearly stable and satisfy the compatibility condition. By $\Pi_{X;Y}$ we denote the linear component of this multi-ideal, that is, the linear operators $u \colon E \longrightarrow F$ that send sequence in $X(E)$ to sequences in $Y(F)$. Then $(\Pi_{X;Y},\|\cdot\|_{X;Y})$ is a Banach operator ideal. 

\begin{definition}\rm \cite{prims} We say that an $n$-linear operator $A \in {\cal L}(E_1, \ldots, E_n;F)$ belongs to $(\Pi_{X;Y} \circ {\cal L})(E_1, \ldots, E_n;F)$ if there are a Banach space $G$, an $n$-linear operator $B  \in {\cal L}(E_1, \ldots, E_n;G)$ and a linear operator $u \in \Pi_{X;Y}(G;F)$ such that $A = u \circ B$. In this case we define
$$\|A\|_{\Pi_{X;Y} \circ {\cal L}} = \inf\{\|B\|\cdot \|u\|_{X;Y} : A = u \circ B, u \in  \Pi_{X;Y}\}. $$
And we say that a polynomial $P \in {\cal P}(^nE;F)$ belongs to $(\Pi_{X;Y} \circ {\cal P})(^nE;F)$ if there are a Banach space $G$, a polynomial  $Q  \in {\cal P}(^n E;G)$ and a linear operator $u \in \Pi_{X;Y}(G;F)$ such that $P = u \circ Q$, or, equivalently, if $\check{P} \in (\Pi_{X;Y} \circ {\cal L})(^n E;F)$. In this case we define
$$\|P\|_{\Pi_{X;Y} \circ {\cal L}} = \inf\{\|Q\|\cdot \|u\|_{X;Y} : P = u \circ Q, u \in  \Pi_{X;Y}\}. $$
\end{definition}

We start by proving that, in the cases that matter to us, the ideals we work with in this paper contain these composition ideals.

\begin{proposition}\label{xycomp} Let $X$ and $Y$ be sequence classes with $X$ multilinearly stable. Then $\Pi_{X;Y} \circ {\cal L} \subseteq {\cal L}_{X;Y}$ with $ {\|\cdot\|}_{X;Y}\leq\|\cdot\|_{\Pi_{X;Y}\circ {\cal L}}$ and $\Pi_{X;Y} \circ {\cal P} \subseteq {\cal P}_{X;Y}$ with $ {\|P\|}_{X;Y;2}\leq \frac{n^n}{n!}\cdot \|P\|_{\Pi_{X;Y}\circ {\cal P}}$ for every $P \in (\Pi_{X;Y} \circ {\cal P})(^nE;F)$. Moreover, is $X$ is polynomially stable, then $ {\|P\|}_{X;Y;1}\leq  \|P\|_{\Pi_{X;Y}\circ {\cal P}}$.
\end{proposition}

\begin{proof} Given $A \in (\Pi_{X;Y} \circ {\cal L})(E_1,\ldots, E_n;F)$, there are a Banach space $G$,  $u \in \Pi_{X;Y}(G;F)$ and $B \in {\cal L}(E_1,\ldots, E_n;G)$ such that que $A = u \circ B$. Since $X$ is multilinearly stable and $u$ is $(X;Y)$-summing, it follows that %Para todas sequências $(x_j^k)_{j=1}^\infty \in X(E_k)$, $k = 1, \ldots, n$, $(B(x_j^1, \ldots, x_j^n))_{j=1}^\infty \in X(G)$ pois $X$ é multilinearmente estável. E como $u$ é $(X;Y)$-somante,
%$$(A(x_j^1, \ldots, x_j^n))_{j=1}^\infty = (u(B(x_j^1, \ldots, x_j^n)))_{j=1}^\infty  \in Y(F), $$
%provando que 
$A \in {\cal L}_{X;Y}(E_1, \ldots, E_n;F)$. Furthermore,
\begin{align*} \|A\|_{X;Y} & = \|\widetilde{A}\| = \sup\{\|(A(x_j^1, \ldots, x_j^n))_{j=1}^\infty\|_{Y(F)} : (x_j^k)_{j=1}^\infty \in B_{X(E_k)}, k = 1, \ldots, n\}\\
&= \sup\{\|(u(B(x_j^1, \ldots, x_j^n)))_{j=1}^\infty\|_{Y(F)} : (x_j^k)_{j=1}^\infty \in B_{X(E_k)}, k = 1, \ldots, n\}\\
&= \sup\{\left\|\widetilde{u}\right((B(x_j^1, \ldots, x_j^n))_{j=1}^\infty)\left)\right\|_{Y(F)} : (x_j^k)_{j=1}^\infty \in B_{X(E_k)}, k = 1, \ldots, n\}\\
& \leq \|\widetilde{u}\|_{{\cal L}(X(G);Y(F))} \cdot \sup\{\left\|(B(x_j^1, \ldots, x_j^n))_{j=1}^\infty)\right\|_{X(G)} : (x_j^k)_{j=1}^\infty \in B_{X(E_k)}, k = 1, \ldots, n\}\\
& = \|u\|_{X;Y} \cdot \sup\left\{\mbox{\large{$\|$}}\widetilde{B}\left((x_j^1)_{j=1}^\infty, \ldots, (x_j^n)_{j=1}^\infty\right))\mbox{\large{$\|$}}_{X(G)} : (x_j^k)_{j=1}^\infty \in B_{X(E_k)}, k = 1, \ldots, n\right\}\\
&=  \|u\|_{X;Y} \cdot \mbox{\large{$\|$}}\widetilde{B} \colon X(E_1) \times \cdots \times X(E_k) \longrightarrow X(G)\mbox{\large{$\|$}} =  \|u\|_{X;Y}\|\cdot \|B\|.
\end{align*}
Taking the infimum over all such factorizations of $A$ we get $ \|A\|_{X;Y} \leq \|A\|_{\Pi_{X;Y}\circ {\cal L}}$.

In the chain
\begin{align*}P \in \Pi_{X;Y} \circ {\cal P}(^mE;F) &\Longrightarrow  \check{P} \in \Pi_{X;Y} \circ {\cal L}(^mE;F)\Longrightarrow \check{P}  \in {\cal L}_{X;Y}(^mE;F)\\
& \Longrightarrow P \in {\cal P}_{X;Y}(^mE;F),
\end{align*}
the first implication follows from \cite{prims}, the second from Proposition  \ref{xycomp} and the third from Theorem \ref{equivhippol}. Take a Banach space $G$, $Q \in {\cal P}(^nE;G)$ and $u \in \Pi_{X;Y}(G;F)$ such that $P = u \circ Q$. Since ${\cal L}_{X;Y}$ is a Banach hyper-ideal (Theorem \ref{L_XYhiperideal}), we have
$$\|P\|_{X;Y;2} = \|\check{P}\|_{X;Y} = \|u \circ \check{Q}\|_{X;Y} \leq \|u\|_{X;Y}\cdot \|\check{Q}\| \leq \frac{n^n}{n!}\cdot \|u\|_{X;Y}\cdot \|Q\|. $$
Taking the infimum over all such factorizations of $P$ we get the first polynomial norm inequality.

For $X$ polynomially stable, the second polynomial norm inequality follows similarly to the multilinear case proved above. %Suponha agora que $X$ seja polinomialmente estável. Para toda fatoração $P = u \circ Q$ com $u$ em $\Pi_{X;Y}$,
%\begin{align*} \|P\|_{X;Y;1} & = \|\widetilde{P}\| = \sup\{\|(P(x_j))_{j=1}^\infty\|_{Y(F)} : (x_j)_{j=1}^\infty \in B_{X(E)}\}\\
%&= \sup\{\|(u(Q(x_j)))_{j=1}^\infty\|_{Y(F)} : (x_j)_{j=1}^\infty \in B_{X(E)}\}\\
%&= \sup\{\left\|\widetilde{u}\right((Q(x_j))_{j=1}^\infty)\left)\right\|_{Y(F)} : (x_j)_{j=1}^\infty \in B_{X(E)}\}\\
%& \leq \|\widetilde{u}\|_{{\cal L}(X(G);Y(F))} \cdot \sup\{\left\|(Q(x_j))_{j=1}^\infty)\right\|_{X(G)} : (x_j)_{j=1}^\infty \in B_{X(E)}\}\\
%& = \|u\|_{X;Y} \cdot \sup\left\{\left\|\widetilde{Q}\left((x_j)_{j=1}^\infty)\right)\right\|_{X(G)} : (x_j)_{j=1}^\infty \in B_{X(E)}\right\}\\
%&=  \|u\|_{X;Y} \cdot \left\|\widetilde{Q} \colon X(E) \longrightarrow X(G)\right\| =  \|u\|_{X;Y}\|\cdot \|Q\|.
%\end{align*}
%Tomando o ínfimo sobre todas essas fatorações obtemos $ \|P\|_{X;Y;1} \leq \|P\|_{\Pi_{X;Y}\circ {\cal P}}$.
\end{proof}

We finish the paper by providing an example where $\Pi_{X;Y} \circ {\cal L} \neq {\cal L}_{X;Y}$ and $\Pi_{X;Y} \circ {\cal P} \neq {\cal P}_{X;Y}$ with $X$ multilinearly stable.

\begin{example}\rm For $p \geq 1$, by $\Pi_p$ we denote the Banach ideal of absolutely $p$-summing operators, that is, %o ideal de Banach dos opeardores lineares absolutamente $p$-somantes, ou seja, operadores que transformam sequências fracamente $p$-somáveis em sequências absolutamente $p$-somáveis (veja \cite{djt}). Na nossa notação, 
$\Pi_p = \Pi_{\ell_p^w(\cdot); \ell_p(\cdot)}$. %Mostremos que
%\begin{equation}\label{popu}{\cal L}_{\ell_p^w(\cdot); \ell_p(\cdot)} \not\subseteq  \Pi_{\ell_p^w(\cdot); \ell_p(\cdot)} \circ {\cal L} = \Pi_p \circ {\cal L}.
%\end{equation}
Let $E$ be a nonreflexive Banach space. Choose $\varphi \in E^*$, $\|\varphi\| = 1$, and consider the continuous bilinear operator
$$A \colon E \times E \longrightarrow E~,~A(x,y) = \varphi(x)y. $$
On the one hand, given sequences $(x_j)_{j=1}^\infty, (y_j)_{j=1}^\infty \in \ell_p^w(E)$, taking $K > 0$ such that $\|y_j\| \leq K$ for every $j$ (weakly $p$-summable sequences are bounded), we have
\begin{align*}\sum_{j=1}^\infty \|A(x_j,y_j)\|^p = \sum_{j=1}^\infty |\varphi(x_j)|^p\cdot \|y_j\|^p  \leq K^p \cdot \sum_{j=1}^\infty |\varphi(x_j)|^p \leq K^p \cdot \sup_{\psi \in B_{E^*}}\sum_{j=1}^\infty |\psi(x_j)|^p  <  \infty,
\end{align*}
showing that $(A(x_j, y_j))_{j=1}^\infty \in \ell_p(E)$; hence $A \in {\cal L}_{\ell_p^w(\cdot); \ell_p(\cdot)}(^2E;E)$.

On the other hand, from \cite[Proposition p.\,461]{jmaa2002} we know that $A$ is not weakly compact. By \cite{teseryan} the linearization of $A$ on the projective tensor product $E \widehat{\otimes}_\pi E$ is a non-weakly compact linear operator, hence a non-absolutely $p$-summing linear operator by \cite[Theorem 2.17]{djt}. By \cite[Proposition 3.2]{prims} it follows that  que $A \notin ({\Pi_p} \circ {\cal L})(^2E;E)$, proving that ${\cal L}_{\ell_p^w(\cdot); \ell_p(\cdot)} \neq  \Pi_p \circ {\cal L}=  \Pi_{\ell_p^w(\cdot); \ell_p(\cdot)} \circ {\cal L} $. In particular, ${\cal L}_{\ell_p^w(\cdot); \ell_p(\cdot)}$ is not a composition ideal. In the case $p = 1$ we have $\ell_1^w(\cdot)$ multilinearly stable and ${\cal L}_{\ell_1^w(\cdot); \ell_1(\cdot)}$ % \neq \Pi_{\ell_1^w(\cdot); \ell_1(\cdot)} \circ {\cal L} $, showing that ${\cal L}_{\ell_1^w(\cdot); \ell_1(\cdot)}$ 
is not a composition multi-ideal.

For polynomials, taking $P = \widehat{A}$, Theorem  \ref{equivhippol} gives $P \in {\cal P}_{\ell_p^w(\cdot); \ell_p(\cdot)}(^2E;E)$ and the same argument of the bilinear case gives $P \notin ({\Pi_p} \circ {\cal P})(^2E;E)$. This proves that ${\cal P}_{\ell_p^w(\cdot); \ell_p(\cdot)} \neq \Pi_{\ell_p^w(\cdot); \ell_p(\cdot)} \circ {\cal P} $. In the case $p = 1$ we have $\ell_1^w(\cdot)$ multilinearly stable and ${\cal P}_{\ell_1^w(\cdot); \ell_1(\cdot)}$ is not a composition polynomial ideal.
\end{example}

\bigskip

\noindent Faculdade de Matem\'atica~~~~~~~~~~~~~~~~~~~~~~Instituto de Matem\'atica e Estat\'istica\\
Universidade Federal de Uberl\^andia~~~~~~~~ Universidade de S\~ao Paulo\\
38.400-902 -- Uberl\^andia -- Brazil~~~~~~~~~~~~ 05.508-090  -- S\~ao Paulo -- Brazil\\
e-mail: botelho@ufu.br ~~~~~~~~~~~~~~~~~~~~~~~~~e-mail: raquelwood@ime.usp.br

\end{document}